\documentclass[a4paper,12pt]{article}
\usepackage[cp1251]{inputenc}
\usepackage{amssymb,amsmath,amsthm}
\usepackage{graphics,graphicx,color}
\usepackage{srcltx}
\DeclareGraphicsExtensions{.eps} \textwidth=169mm
\definecolor{darkred}{rgb}{0.9,0.1,0.1}

\newtheorem{theorem}{Theorem}[section]
\newtheorem{lemma}{Lemma}[section]

\newtheorem{remark}{Remark}[section]
{\rm}
\definecolor{darkred}{rgb}{0.9,0.1,0.1}

\def\eps{\varepsilon}
\def\eps{\varepsilon}
\begin{document}

\title{Periodic homogenization of convolution type operators with heavy tails}

\author{ Andrey Piatnitski$^{1,2}$%\footnote{Corresponding author, ORCID
%0000-0002-9874-6227}
\ , Elena Zhizhina$^{1,2}$}{\small

 \maketitle

    \begin{quote}
1) The Arctic University of Norway, campus Narvik, PO Box 385, Narvik 8505, Norway
\\[2mm]
2) Higher School of Modern Mathematics MIPT, 1st Klimentovskiy per.,  Moscow, Russia, 115184.
\\[2mm]
Emails:  apiatnitski@gmail.com, elena.jijina@gmail.com
\end{quote}}

\begin{abstract}
  The paper deals with periodic homogenization of nonlocal symmetric convolution type operators in $L^2(\mathbb R^d)$,
  whose kernel is the product of a density that belongs to the domain of attraction of an $\alpha$-stable law
  and a rapidly oscillating positive periodic function.
   Assuming that  the local oscillation of the said density satisfies a proper upper bound at infinity,
we prove homogenization result for the studied family of operators.
\end{abstract}

\section{Introduction}
The paper deals with homogenization problem for nonlocal convolution-type operators
whose convolution kernel $p(x-y)$ is non-negative, symmetric $p(x-y) = p(y-x)$ and has heavy tails so that the second moment of this kernel is not finite
on the one hand, and this kernel is not of the form  $|x-y|^{-d-\alpha}$,  on the other hand.
The corresponding operator reads
\begin{equation}\label{intr_Le}
L^{\eps} u(x) = \frac{1}{\eps^{d+\alpha}} \int_{\mathbb{R}^d} p\big(\frac{x-y}{\eps}\big) \Lambda \big(\frac{x}{\eps}, \frac{y}{\eps} \big) \big( u(y) - u(x) \big) dy, \quad 0<\alpha < 2, \quad \eps \in (0,1),
\end{equation}
here $\eps$ is a small positive parameter, and our goal is to study the asymptotic behaviour of the resolvent
$(-L^\eps+m)^{-1}$ as $\eps\to0$ for a fixed $m>0$.
We assume that $p(z)\in L^1(\mathbb R^d)$ and that the  function $\frac1{\eps^{d+\alpha}}p\big(\frac z\eps\big)$  approximates at infinity in a weak sense a function of the form $k\big(\frac z{|z|}\big) |z|^{-d-\alpha}$,
where $k(s)$ is a positive symmetric continuous function on the sphere $S^{d-1}$.
It should be noted that $p(z)$ may show rather irregular behaviour in the vicinity of zero.
The coefficient $\Lambda(x,y) = \Lambda(y,x)$ is positive, symmetric, bounded and periodic in both variables.
Under these assumptions and  an additional assumption that the local oscillation of $p(z)$ decays at infinity faster than $p(z)$, we show that the family of operators $-L^\eps+m$ admits homogenization, that is the resolvent   $(-L^\eps+m)^{-1}$
converges strongly in $L^2(\mathbb R^d)$ as $\eps\to0$ to the resolvent $(-L^{\rm eff}+m)^{-1}$ of the effective operator $L^{\rm eff}$ that reads
\begin{equation}\label{intr_L_eff}
L^{\rm eff} u(x) =  {\rm p. \, v.} \, \int_{\mathbb{R}^d} \frac{\overline{\Lambda}k\big(\frac{x-y}{|x-y|}\big)\big( u(y) - u(x) \big)}
{|x-y|^{d+\alpha}} dy,
\end{equation}
where $\overline{\Lambda}>0$ is the mean value of $\Lambda(\cdot)$.

Notice that our  main result remains valid if  $p(z) \in L^1(\mathbb{R}^d)$  is comparable at infinity with the function
$\frac{L(|z|)}{|z|^{d+\alpha}}$, where $L(r)$ is a slowly varying function, see Remark \ref{SV} for the details.

 It is interesting to observe that, while the operators $L^\eps$ are bounded in $L^2(\mathbb{R}^d)$, the limit operator $L^{\rm eff}$ is unbounded and has a domain $D(L^{\rm eff})\subset H^\frac\alpha2(\mathbb R^d)$.
  %that is dense in $L^2(\mathbb{R}^d)$.

In the existing literature the homogenization problem for operators
\begin{equation}\label{prev_eq}
L^\eps  u(x) =  {\rm p. \, v.} \, \int_{\mathbb{R}^d} \frac{ \Lambda \big(\frac{x}{\eps}, \frac{y}{\eps} \big) \big( u(y) - u(x) \big) }{|x-y|^{d+\alpha}}dy
\end{equation}
with a bounded positive coefficient $\Lambda$ was studied in \cite{KPZ}. It was shown that both for periodic
and statistically homogeneous symmetric functions $\Lambda$ the homogenization result holds for the family $\{L^\eps\}$ as $\eps\to0$, and the effective operator is of the form \eqref{intr_L_eff} with $k(s)=1$ and $\overline{\Lambda}>0$ being the mean
value of $\Lambda$. Under additional regularity assumptions, similar result holds in the case of non-symmetric periodic $\Lambda$, if $0<\alpha<1$. However,
in this case the effective coefficient is not the average of $\Lambda$, it depends on the kernel of the adjoint periodic operator.

If $1\leqslant\alpha<2$, the homogenization result for the corresponding parabolic equations holds in moving coordinates.
For operators of the (non-divergence) form
$$
L^{\rm eff} u(x) =  \int_{\mathbb{R}^d} \frac{\Lambda \big(\frac{x}{\eps}, \frac{y}{\eps} \big)
\big( u(y) - u(x)- \mathbf{1}_{\{|x-y|<\eps\}}\nabla u(x) \big)}
{|x-y|^{d+\alpha}} dy
$$
 with non-symmetric kernels
this result was established in \cite{CCKW1}.

The homogenization problems for symmetric pure jump d-dimensional L$\acute{e}$vy processes were studied in \cite{SU}, where the limit process has been defined using Mosco convergence.

% Results concerning the convergence rates of the
An important issue in homogenization of  stable-like operators in periodic environments is obtaining estimates for the rate of convergence. This issue was partially addressed
% homogenization for stable-like operators in periodic environments have been under consideration
 in the recent works \cite{PSSZ} and \cite{CCKW2}.
The work \cite{CCKW2} focuses on quantitative homogenization of symmetric stable-like operators defined in \eqref{prev_eq}.
 The authors consider the equation $-L^\eps u+mu=h$ and  for the right-hand sides $h$ satisfying certain regularity conditions obtain sharp in order estimates
for the rate of resolvent convergence in the strong norm.
Similar sharp in order estimates in the operator norm for the generic right-hand side and the generic coefficient satisfying the uniform ellipticity condition
were established  in \cite{PSSZ}.
% the analogous estimates have been proved in the operator norm and under more general assumptions on the coefficients of the operator.

 It is worth noting that in the existing works the authors studied the homogenization of stable-like processes. In the present paper we consider the convolution type operators with integrable kernels that belong to the domain of attraction of some stable law.
 In \cite{PZ17} we studied the homogenization problem for convolution type operators whose jumping kernels  have finite second moment. Using the corrector technique we proved that the limiting operator is a second order elliptic differential  operator with a constant positive definite effective matrix. Therefore, in this case the limiting process is diffusive. In the present work we also study convolution type operators, but now the jumping kernel has infinite second moment and the corresponding measure $p(z)dz$ belongs to the domain of attraction of a stable law, see e.g. \cite{Kol}, \cite{UZ} for further details. We prove that in this case, the effective process is a symmetric $\alpha$ - stable Levy process.
Unlike to the paper \cite{PZ17}, where the corrector technique  was exploited, the approach used here relies on the compactness arguments.

The paper consists of Introduction and two sections. In the first section we provide a problem setup and formulate
our main homogenization results.
The second section is devoted to the proof of the main result. We obtain a priori estimates, establish compactness
results, and, in Subsection \ref{ss_conv_loc}, prove the strong resolvent convergence in the $L^2_{\rm loc}(\mathbb R^d)$ topology.
Finally, in Subsection \ref{ss_conv}, we show that the convergence takes place in $L^2(\mathbb R^d)$ topology.

\section{Problem Setup and Main Theorem}

We consider an operator of the form
\begin{equation}\label{Le}
L^{\eps} u(x) = \frac{1}{\eps^{d+\alpha}} \int_{\mathbb{R}^d} p\big(\frac{x-y}{\eps}\big) \Lambda \big(\frac{x}{\eps}, \frac{y}{\eps} \big) \big( u(y) - u(x) \big) dy, \quad 0<\alpha < 2, \quad \eps \in (0,1),
\end{equation}
where  $ p (z)\in L^1(\mathbb{R}^d)$ is a non-negative function that satisfies the following symmetry  condition: $p(z)= p(-z)$ for all $z\in\mathbb R^d$. Without loss  of generality we assume that
 $\int_{\mathbb{R}^d} p(z) dz = 1$.  Then
 $p(z)$ is the density of a symmetric distribution on $\mathbb R^d$.

Let $k: \, S^{d-1} \to \mathbb{R}_+$ be a continuous symmetric positive function: $k(-s) = k(s)$ and $k(s)>0$ for
all $s\in S^{d-1}$; $d_S$ denotes the Lebesgue measure on the sphere $S^{d-1}$.
We assume that the measure with density $p(\cdot)$ belongs to the domain of attraction of a symmetric $\alpha$-stable law, that means
\begin{equation}\label{pe}
\min\{ 1, \, |z|^2 \} \, \frac{1}{\eps^{d+\alpha}} \, p\big(\frac{z}{\eps}\big) \, dz \ \to \ \min\{ 1, \, |z|^2 \} \, \frac{ d|z|}{|z|^{1+\alpha}} \, k (\tilde z) \, d_S \tilde z, \quad   \eps \to 0, \quad \tilde z = \frac{z}{|z|} \in S^{d-1},
\end{equation}
where convergence is in the weak sense, see e.g.  \cite{Kol} (Sect. 8.3).  Remind, that the weak convergence of measures $\mu_n \to \mu$ as $n \to \infty$ means that $(f, \mu_n) \to (f, \mu)$ for any $f \in C_b(\mathbb{R}^d)$.

We suppose in this paper that the density $p(z) \in  L^1(\mathbb{R}^d)$ satisfies the following conditions:
\begin{itemize}
\item[1)]
\begin{equation}\label{bas_prop_p}
p(z)\geqslant 0, \quad  p(-z)=p(z) \hbox{ for all } z\in\mathbb R^d, \quad  \int_{\mathbb{R}^d} p(z) dz = 1,
\end{equation}
\item[2)] for almost all $z$ such that $|z|\geqslant M$
\begin{equation}\label{M-1}
\frac{\beta_1}{|z|^{d+\alpha}} \le p(z) \le \frac{\beta_2}{|z|^{d+\alpha}}, % \qquad |z|>M,
\end{equation}
with positive constants $\beta_1, \; \beta_2$ and $M \ge 1$.
\item[3)]
For an arbitrary open subset $\Omega$ of $S^{d-1}$ with a boundary of Lebesgue measure zero:
\begin{equation}\label{M-2}
\int_{|z|>n} \int_{\tilde z \in \Omega} p(z) \, dz \sim \frac{1}{\alpha \, n^\alpha} \int_\Omega k(s) \, ds, \quad n \to \infty,
\end{equation}
where  $\tilde z = \frac{z}{|z|} \in \Omega \subset S^{d-1}$, and the symbol $"\sim"$ means that the ratio of the two sides of this formula tends to one as $n \to \infty$.
\item[4)]  There exists a constant $K>0$ such that
\begin{equation}\label{M-3}
\mathop{\rm ess\,sup}\limits_{{|\gamma| \le K}\atop{|z|\geqslant r}} \frac{|p(z+\gamma) - p(z)|}{p(z)} \, \to \, 0 \quad \mbox{as } \quad r \to \infty,
 \end{equation}
% with a given constant. % that will be chosen later.
\end{itemize}
It is convenient to introduce a function $\phi_K(r)$ defined by
\begin{equation}\label{def_phidec}
  \phi_K(r)=%\sup\limits_{|z|\geqslant r}\sup\limits_{|\gamma| \le K}
  \mathop{\rm ess\,sup}\limits_{{|\gamma| \le K}\atop{|z|\geqslant r}}  \frac{|p(z+\gamma) - p(z)|}{p(z)}.
\end{equation}
According to \eqref{M-3} %and \eqref{def_phidec}
the function $ \phi_K(\cdot)$ decreases and tends to zero at infinity.

Notice that condition $p  \in  L^1(\mathbb{R}^d)$ implies that the corresponding measure is absolutely continuous w.r.t. the Lebesgue measure and has no atoms.   Observe also that  relation  \eqref{M-3} holds for any $K>0$ if and only if
it holds for some $K>0$.  %3) implies that the holds with any constant $K>0$.

It follows from \eqref{M-1} and \eqref{M-2} that
\begin{equation}\label{estbeta}
  \beta_1\leqslant k(s)\leqslant \beta_2 \quad\hbox{for all }s\in S^{d-1}.
\end{equation}

If the density $p$ satisfies condition \eqref{M-2}, then \eqref{pe} holds, see Proposition 8.3.1, \cite{Kol}. The additional conditions \eqref{M-1} and \eqref{M-3} are imposed in order to make the homogenization result hold.

%has a "star shape"  at infinity, i.e. there exist a bounded positive function $0<k_1\le k(s)\le k_2< \infty, \ s %\in S^{d-1},$ and a constant $M>0$ such that $p(\cdot)$ can be represented as
%\begin{equation*}
%p(z) \  = \ \frac{K(\frac{z}{|z|}, |z|)}{|z|^{d+\alpha}} \quad \mbox{for all } \; |z|>M,
%\end{equation*}
%where
%\begin{equation*}
%0<\tilde k_1\le K(s,r)\le \tilde k_2< \infty \quad \mbox{and } \quad \sup\limits_{s \in S^{d-1}} |K(s, r) - k(s)| %\to 0 \quad \mbox{as } \ r \to \infty.
%\end{equation*}

About $\Lambda$ we assume that $ \Lambda (x, y) = \Lambda (y, x)$ is a symmetric periodic function  such that
\begin{equation}\label{Lambda}
0 < \gamma_1 \le \Lambda(x,y) \le \gamma_2 < \infty.
\end{equation}
Without loss of generality we assume $\Lambda(x,y)$ has a period $[0,1)^{2d}$. In what follows  we identify the period of $\Lambda$ with the torus $\mathbb{T}^{2d}$.
Then $-L^\eps$ for every $\eps \in (0,1)$ is a bounded positive self-adjoint operator in $L^2(\mathbb{R}^d)$.
% with domain ${\cal D}(L^\eps) = L^2(\mathbb{R}^d)$.

For $m>0$ denote by $u^{\eps}  \in   L^2(\mathbb{R}^d)$ the solution of equation
\begin{equation}\label{Ue}
-L^{\eps} u^{\eps} + m u^{\eps} = f \quad \mbox{with } \; f \in L^2(\mathbb{R}^d), %\; \; m>0,
\end{equation}
and by $u  \in   L^2(\mathbb{R}^d)$ the solution of equation
\begin{equation}\label{U-limit}
-L^{0} u  + m u = f \quad \mbox{with the same } \; f \in L^2(\mathbb{R}^d),% \; \; m>0,
\end{equation}
where
\begin{equation}\label{L0-limit}
L^{0} u(x) = \int_{\mathbb{R}^d} \Lambda^{\rm eff}(x,y) \, \frac{( u(y) - u(x))}{|y-x|^{d+\alpha}} dy, \quad 0<\alpha < 2;
\end{equation}
with
\begin{equation}\label{L0limit}
\Lambda^{\rm eff}(x,y) = \overline{\Lambda} \, k\big( \frac{x-y}{|x-y|} \big), \qquad \overline{\Lambda} = \int_{\mathbb{T}^d}\int_{\mathbb{T}^d} \Lambda(x,y) dx dy.
\end{equation}
It should be emphasized that the operator $-L^0$ is an unbounded non-negative self-adjoint operator in $L^2 (\mathbb R^d)$.
It corresponds to the quadratic form
$$
a^0(u)=\frac12 \int_{\mathbb{R}^d} \int_{\mathbb{R}^d} \Lambda^{\rm eff}(x,y) \, \frac{( u(y) - u(x))^2}{|y-x|^{d+\alpha}} dydx.
$$
Due to \eqref{estbeta} this form is comparable to the form
$\int_{\mathbb R^d}\frac{(u(x)-u(y))^2}{|x-y|^{d+\alpha}}dx$. Therefore, the form $a^0(u)$ with the domain $H^{\frac\alpha2}(\mathbb R^d)$ is closed.  As a consequence, the  domain $D(-L^0)$ of $-L^0$ belongs to $H^\frac\alpha2(\mathbb R^d)$ and is dense in this space  (and in  $L^2(\mathbb{R}^d)$).

Our main result is the following theorem.
\begin{theorem}\label{T1}
Let  conditions \eqref{bas_prop_p} - \eqref{M-3} be fulfilled,  and assume that \eqref{Lambda} holds true. Then for each $f \in L^2(\mathbb{R}^d)$ the solution $ u^{\eps}$ of \eqref{Ue} converges strongly in $L^2(\mathbb{R}^d)$ to the solution $u$ of \eqref{U-limit} - \eqref{L0limit}.
\end{theorem}
\medskip

\begin{remark}\label{SV}
 The statement of Theorem \ref{T1} remains valid if  $p(z) \in L^1(\mathbb{R}^d)$  is comparable at infinity with the function
$\frac{L(|z|)}{|z|^{d+\alpha}}$, where $L(r)$ is a slowly varying function. Recall that a positive function $L(r)$, defined for $r \ge 0$, is said to be slowly varying if, for all $g>0$, $\lim_{r \to \infty} \frac{L(r g)}{L(r)} =1$.
It is known, see e.g. \cite{Feller}, that the measures $p(z) dz$
%of the form $\frac{L(|z|)}{|z|^{d+\alpha}}$, where $L(r)$ is a slowly varying function,
belong to the domain of attraction of $\alpha$-stable law. The corresponding distributions converge to the $\alpha$-stable distribution when scaled with a factor of  $a_n = n^\alpha \, L(n)$, not $n^\alpha$. Thus,  we have to modify the definition of operator $L^\eps$ taking in \eqref{intr_Le} and in \eqref{Le} the scaling factor $\frac{1}{\eps^{d+\alpha} \, L(\frac1\eps)}$ instead of $\frac{1}{\eps^{d+\alpha}}$. Also,  assumptions \eqref{M-1} - \eqref{M-2} should be rearranged in this case as follows:
\begin{equation}\label{M-1bis}
\beta_1 \frac{L(|z|)}{|z|^{d+\alpha}} \le p(z)  \le  \beta_2 \frac{L(|z|)}{|z|^{d+\alpha}},  \qquad |z|>M;
\end{equation}
%For an arbitrary open subset $\Omega$ of $S^{d-1}$ with a boundary of Lebesgue measure zero:
\begin{equation}\label{M-2bis}
\int_{|z|>n} \int_{\tilde z \in \Omega} p(z) \, dz \ \sim \  \frac{L(n)}{\alpha \, n^\alpha} \int_\Omega k(s) \, ds, \quad n \to \infty.
\end{equation}
Then the arguments used  in the proof of Theorem \ref{T1} remain valid, and for  the functions $p(z)$ satisfying assumptions \eqref{bas_prop_p}, \eqref{M-3}, \eqref{M-1bis}, \eqref{M-2bis}, the statement of Theorem \ref{T1} holds.
\end{remark}

\section{Proof of the theorem}

\subsection{A priory estimates}

\begin{proof} We start the proof of the theorem with a priory estimates.
Multiplying equation \eqref{Ue} by $u^{\eps}$ and integrating the resulting relation over $\mathbb{R}^d$ we
obtain
$$
m(u^{\eps}, u^{\eps}) = (f, u^{\eps}) + ( L^{\eps} u^{\eps},u^{\eps}).
$$
Since
\begin{equation}\label{C1-bef}
 (-L^{\eps} u^{\eps}, u^{\eps}) =\frac{1}{2\eps^{d+\alpha}} \int_{\mathbb{R}^d} p\big(\frac{x-y}{\eps}\big) \Lambda \big(\frac{x}{\eps}, \frac{y}{\eps} \big) \big( u(y) - u(x) \big)^2 dy\geqslant 0,
 \end{equation}
then
\begin{equation}\label{C1}
 \|u^\eps\| \le \frac1m \| f\| =: C_1,
\end{equation}
with a constant $C_1 = C(f)$ that does not depend on $\eps$, and in what follows we will use the notation  $\| \cdot \| = \| \cdot \|_{L^2(\mathbb{R}^d)}$.
Thus the family of functions $\{ u^\eps\}$ is bounded in $ L^2(\mathbb{R}^d)$.
Moreover,  \eqref{Ue} and bound \eqref{C1} yield
\begin{equation}\label{C2}
\big(-L^{\eps} u^\eps, u^\eps \big) \le \|f\| \, \|u^\eps\| \le \frac1m \| f\|^2 =: C_2,
\end{equation}
and inequality \eqref{C2} together with \eqref{Lambda} and \eqref{C1-bef} imply that
\begin{equation}\label{C3}
\frac{1}{\eps^{d+\alpha}} \int\limits_{\mathbb{R}^d} \int\limits_{\mathbb{R}^d} p \big(\frac{x-y}{\eps}\big)  \big( u^\eps(y) - u^\eps(x) \big)^2 dx dy \le C_3.
\end{equation}
In the proof of estimates \eqref{C1-bef} - \eqref{C2} we use only the symmetry of functions $p(z)$ and $\Lambda(x,y)$.

Moreover, inequality \eqref{C3} together with a lower bound in \eqref{M-1} yield  the following uniform in $\eps$ estimate:
\begin{equation}\label{C4}
\begin{array}{l}
\displaystyle
 \int\limits_{|x-y|> M\eps} \frac{( u^\eps(y) - u^\eps(x))^2}{|x-y|^{d+\alpha}} dx dy = \frac{1}{\varepsilon^{d+\alpha}}\int\limits_{\mathbb{R}^d}  \int\limits_{|z|> M\eps} \frac{( u^\eps(x- z) - u^\eps(x))^2}{\big|\frac{z}{\varepsilon}\big|^{d+\alpha}} dx dz
 \\[4mm] \displaystyle
 \le \frac{\beta_1^{-1}}{\varepsilon^{d+\alpha}}\int\limits_{\mathbb{R}^d}  \int\limits_{|x-y|> M\eps}  p \big(\frac{x-y}{\eps}\big) \, ( u^\eps(y) - u^\eps(x))^2 \, dx dy  \le C_4,
 \end{array}
\end{equation}
where $M$ is the same constant as in \eqref{M-1}.

\subsection{Compactness results}
 
We are going to show that any sequence  $\{ u^{\eps_j}\}$, $\eps_j \to 0$ as $j\to\infty$,  is compact in $ L^2_{\rm loc}(\mathbb{R}^d)$. 
As follows from the Kolmogorov-Riesz  compactness theorem, a subset $S$ of $L^2_{\rm loc}(\mathbb{R}^d)$
is compact in $ L^2_{\rm loc}(\mathbb{R}^d)$, if 
%Here we
%use the Riesz criterium for compactness {\color{blue} that gives us the following sufficient condition for compactness  of a subset $S$ of $L^2_{\rm loc}(\mathbb{R}^d)$ }:
\begin{equation}\label{Riesz}
S \subset L^2(\mathbb{R}^d) \; \mbox{is bounded, \  and } \; \lim\limits_{|h| \to 0+} \sup\limits_{f\in S} %\int\limits_{\mathbb{R}^d}
\int\limits_{G}  |f(x+h) - f(x)|^2 dx=0
\end{equation}
for any bounded $G \subset \mathbb{R}^d$.  The boundedness of the family  $\{ u^{\eps_j}\}$ in 
$L^2(\mathbb R^d)$  is a direct consequence of  a priory estimate \eqref{C1}. 
To obtain  the second relation in  \eqref{Riesz} we first prove the following lemma.

\begin{lemma}\label{L1}
%Let $G \subset \mathbb{R}^d$ be an arbitrary bounded domain and
For any $h \in  \mathbb{R}^d$

1) If $|h|>3M \eps$, then
\begin{equation}\label{L1-1}
 \int\limits_{\mathbb{R}^d} (u^\eps(x+h) - u^\eps(x))^2 dx \le c_1 |h|^\alpha.
\end{equation}

2) If $|h|< 3M \eps$, then
\begin{equation}\label{L1-2}
 \int\limits_{\mathbb{R}^d} (u^\eps (x+h) - u^\eps(x))^2 dx \le c_2 \eps^\alpha.
\end{equation}
\end{lemma}

\begin{proof} %In the proof we omit the superscript $\eps$ in the notation $u^\eps$.
Denote by $\Gamma_h$ a cube centered at $\frac{h}{2}$ with a side $\frac{|h|}{3}$. Then
\begin{equation}\label{proof-1}
\frac13|h| \le |z| \le \frac23 |h|, \quad \mbox{for } \; \; z \in \Gamma_h,
\end{equation}
and we have
\begin{equation}\label{PL1-1}
\begin{array}{l}
\displaystyle
 \int\limits_{\mathbb{R}^d} (u^\eps(x+h) - u^\eps(x))^2 dx = \frac{1}{|\Gamma_h|}  \int\limits_{\Gamma_h} \int\limits_{\mathbb{R}^d} (u^\eps(x+h) - u^\eps(x))^2 dx dz
\\[3mm] \displaystyle
\le  \frac{2}{|\Gamma_h|} \int\limits_{\mathbb{R}^d} \int\limits_{\Gamma_h}  (u^\eps(x+h) - u^\eps(x+z))^2  dz dx +  \frac{2}{|\Gamma_h|} \int\limits_{\mathbb{R}^d} \int\limits_{\Gamma_h}  (u^\eps(x+z) - u^\eps(x))^2  dz dx.
\end{array}
\end{equation}
It is clear that both integrals can be estimated in the same way, and we consider the second one. If  $|h|>3M \eps$, then $|z| \ge M \eps$ for all $z\in \Gamma_h$. Using \eqref{PL1-1} together with \eqref{C4} and \eqref{proof-1} we get \eqref{L1-1}:
$$
\frac{2}{|\Gamma_h|} \int\limits_{\mathbb{R}^d} \int\limits_{\Gamma_h}  (u^\eps(x+z) - u^\eps(x))^2  dz dx
\le \frac{C_{d, \alpha}|h|^{d+\alpha}}{|h|^d} \int\limits_{\mathbb{R}^d} \int\limits_{|z|\ge M \eps} \frac{ (u^\eps(x+z) - u^\eps(x))^2}{|z|^{d+\alpha}}  dz dx \le \frac12 c_1 |h|^\alpha,
$$
where constant $c_1$ depends on $d$ and $\alpha$.

If  $|h|< 3M \eps$, we take $h_0=k \eps$ with such $k\in\mathbb R^d$ that $|h_0|> 3M \eps$ and  $|h-h_0|> 3M \eps$, for example, $|h_0|= 7M \eps$. Then using inequality \eqref{L1-1} we obtain \eqref{L1-2}:
$$
\begin{array}{l} \displaystyle
 \int\limits_{\mathbb{R}^d}  (u^\eps(x+h) - u^\eps(x))^2  dx \le  2\int\limits_{\mathbb{R}^d} (u^\eps(x+h) - u^\eps(x+h_0))^2 dx + 2\int\limits_{\mathbb{R}^d} (u^\eps(x+h_0) - u^\eps(x))^2 dx \\[3mm] \le  c_1 |h_0-h|^\alpha + c_1 |h_0|^\alpha \le c_2 \eps^\alpha.
 \end{array}
$$
Lemma is proved.
\end{proof}

The next lemma provides a result on compactness in $L^2_{\rm loc}(\mathbb{R}^d)$ for a sequence $\{u^{\varepsilon_j}\}$ with $\varepsilon_j \to 0$.

\begin{lemma}\label{L2}
Any sequence $\{u^{\varepsilon_j}\}$ with $\varepsilon_j \to 0$ is compact in $L^2_{\rm loc}(\mathbb{R}^d)$.
Moreover, any limit point of this family is an element of $H^\frac\alpha2(\mathbb R^d)$.
\end{lemma}

\begin{proof}
Let us take a sequence $\{u^{\varepsilon_j}\}$ with $\varepsilon_j \to 0$.
Due to \eqref{Riesz} it is sufficient to show that
\begin{equation}\label{comp-1}
\forall \ \varkappa>0 \ \ \exists \ \delta>0 \ \ \hbox{s. t. } \forall \ |h|<\delta \ \mbox{ and } \ \forall \ \varepsilon_j \quad \int\limits_{\mathbb{R}^d}  (u^{\eps_j}(x+h) - u^{\eps_j}(x))^2  dx \le K \varkappa.
\end{equation}

For arbitrary $\varkappa>0$ we put $\delta_1 = 3M \varkappa^{1/\alpha}$  and take $\varepsilon_j$ such that $\varepsilon_j > \frac{\delta_1}{3M} = \varkappa^{1/\alpha}$. Since we have a finite set of such $\varepsilon_j$, then by the Riesz criterium we conclude that
\begin{equation}\label{comp-2}
\forall \ \varkappa>0 \ \ \exists \ \delta_2>0 \ \ \hbox{s. t. } \forall \ |h|<\delta_2  %\varepsilon_j>\varkappa^{1/\alpha}
 \quad \max\limits_{\{j:\varepsilon_j > \varkappa^{1/\alpha}\}} \int\limits_{\mathbb{R}^d}  (u^{\eps_j}(x+h) - u^{\eps_j}(x))^2  dx \le K \varkappa.
\end{equation}
Denote $\delta = \min \{ \delta_1, \ \delta_2 \}$.

1) If $\delta_2 > \delta_1$, then $|h|<\delta_1<\delta_2$.
According \eqref{L1-1} for $\varepsilon_j<\frac{|h|}{3M}$ we have
$$
 \int\limits_{\mathbb{R}^d}  (u^{\eps_j}(x+h) - u^{\eps_j}(x))^2  dx \le C_1 |h|^\alpha<C_1 \delta_1^\alpha=\tilde C_1 \varkappa.
$$
For  $\frac{|h|}{3M}< \varepsilon_j< \varkappa^{1/\alpha}$, using \eqref{L1-2} we get
$$
 \int\limits_{\mathbb{R}^d}  (u^{\eps_j}(x+h) - u^{\eps_j}(x))^2  dx \le C_2 \varepsilon_j^\alpha \le C_2 \varkappa.
$$
2) If $\delta_2 < \delta_1$, then $|h|<\delta_2<\delta_1$. For $\varepsilon_j < \frac{|h|}{3M}$ by \eqref{L1-1} we have
$$
 \int\limits_{\mathbb{R}^d}  (u^{\eps_j}(x+h) - u^{\eps_j}(x))^2  dx \le C_1 |h|^\alpha < C_1 \delta_1^\alpha=\tilde C_1 \varkappa.
$$
For  $\frac{|h|}{3M}< \varepsilon_j< \varkappa^{1/\alpha}$, using \eqref{L1-2} we get
$$
 \int\limits_{\mathbb{R}^d}  (u^{\eps_j}(x+h) - u^{\eps_j}(x))^2  dx \le C_2 \varepsilon_j^\alpha \le C_2 \varkappa.
$$
Thus for all $\varepsilon_j$ estimate \eqref{comp-1} holds.

We turn to the second statement of Lemma. In view of \eqref{M-1} and \eqref{C2} we have
$$
\int_{|x-y|>M\eps}\frac{(u^\eps(x)-u^\eps(y))^2}{|x-y|^{d+\alpha}}dxdy=\frac1{\eps^{d+\alpha}}
\int_{|x-y|>M\eps}\frac{(u^\eps(x)-u^\eps(y))^2}{|\frac{x-y}\eps|^{d+\alpha}}dxdy
$$
$$
\leqslant \int_{\mathbb R^{2d}} \bold{1}\big._{|x-y|>M\eps}(x-y)\frac1{\beta_1}p\Big(\frac{x-y}\eps\Big)
(u^\eps(x)-u^\eps(y))^2 dxdy\leqslant C.
$$
Consider an arbitrary limit point of the family ${u^{\eps_j}}$, denote it $\tilde u$. Then, for a subsequence,
$u^{\eps_j}$ converges to $\tilde u$ almost everywhere in $\mathbb R^d$. This implies that
$$
\bold{1}\big._{|x-y|>M\eps}(x-y)\frac{(u^\eps(x)-u^\eps(y))^2}{|x-y|^{d+\alpha}}
\mathop{\longrightarrow}\limits_{\eps_j\to0} \frac{(\tilde u(x)-\tilde u(y))^2}{|x-y|^{d+\alpha}}
$$
almost everywhere in $\mathbb R^{2d}$. Therefore, by the Fatou lemma,
$$
 \int_{\mathbb R^{2d}} \frac{(\tilde u(x)-\tilde u(y))^2}{|x-y|^{d+\alpha}} dxdy\leqslant C,
$$
which yields $\tilde u\in H^\frac\alpha2(\mathbb R^d)$.

\end{proof}

\begin{remark}
It is worth noting that in the %arguments proving the compactness of the sequence  $\{u^{\varepsilon_j}\}$ as $\varepsilon_j \to 0$,
proof of Lemma \ref{L2}
 we used only the lower bound in condition \eqref{M-1}.
\end{remark}

\subsection{Homogenization in $L^2_{\rm loc}(\mathbb{R}^d)$}\label{ss_conv_loc}

Therefore, for a subsequence, $u^{\varepsilon}$ converges strongly in $L^2_{\rm loc}(\mathbb{R}^d)$ to some function $u$, and the next step of the proof of the theorem is to characterize the function $u$. To do so we follow the same reasoning as in \cite{KPZ} with a suitable adaptation to our case.
We multiply $-L^\eps u^\eps + m u^\eps = f$ by a test function $\varphi \in C_0^\infty(\mathbb{R}^d )$ and integrate over $\mathbb{R}^d $. This yields
\begin{equation}\label{B1}
 \frac{1}{\eps^{d+\alpha}} \int\limits_{\mathbb{R}^d} \int\limits_{\mathbb{R}^d}
 p\big(\frac{x-y}{\eps}\big) \Lambda \big(\frac{x}{\eps}, \frac{y}{\eps} \big) ( u^\eps(y) - u^\eps(x) )
 (\varphi(y) - \varphi(x))dx dy + \int\limits_{\mathbb{R}^d}(m u^\eps \varphi - f \varphi) dx =0.
\end{equation}
Our goal is to pass to the limit as $\eps \to 0$ in \eqref{B1}.
The second integral in \eqref{B1} converges to the integral $\int_{\mathbb{R}^d}(m u \varphi - f \varphi) dx$. To study the first integral in \eqref{B1} we divide the integration over $\mathbb{R}^d \times \mathbb{R}^d$ into three parts:
$$\mathbb{R}^d \times \mathbb{R}^d = G_1^\delta \cup  G_2^\delta \cup  G_3^\delta,
$$
where
$$
G_1^\delta = \{ (x,y): \ |x-y|\ge \delta, \ |x|+|y|\le \delta^{-1} \},
$$
$$
G_2^\delta = \{ (x,y): \ |x-y|\le \delta, \ |x|+|y|\le \delta^{-1} \}, \quad G_3^\delta = \{ (x,y): \  |x|+|y|\ge \delta^{-1} \}.
$$

The integral over $G_2^\delta \cup  G_3^\delta$  for small enough $\delta>0$  is estimated using the Cauchy inequality and estimate \eqref{C3}:
\begin{equation}\label{B2}
\begin{array}{l}
\displaystyle
\Big|  \frac{1}{\eps^{d+\alpha}} \int\limits_{G_2^\delta \cup  G_3^\delta}
 p\big(\frac{x-y}{\eps}\big) \Lambda \big(\frac{x}{\eps}, \frac{y}{\eps} \big) ( u^\eps(y) - u^\eps(x) )
 (\varphi(y) - \varphi(x))dx dy \Big| \\
 \displaystyle
 \le \gamma_2 \Big( \frac{1}{\eps^{d+\alpha}} \int\limits_{G_2^\delta \cup  G_3^\delta}
 p\big(\frac{x-y}{\eps}\big) ( u^\eps(y) - u^\eps(x) )^2 dx dy \Big)^{1/2} \Big( \frac{1}{\eps^{d+\alpha}} \int\limits_{G_2^\delta \cup  G_3^\delta}
 p\big(\frac{x-y}{\eps}\big)
 (\varphi(y) - \varphi(x))^2 dx dy \Big)^{1/2} \\ \displaystyle
  \le \tilde C_1 \,  \Big( \frac{1}{\eps^{d+\alpha}} \int\limits_{G_2^\delta \cup  G_3^\delta}
 p\big(\frac{x-y}{\eps}\big)
 (\varphi(y) - \varphi(x))^2 dx dy \Big)^{1/2}.
%  \qquad  \Big( \int\limits_{G_2^\delta \cup  G_3^\delta} \frac{ (\varphi(y) - \varphi(x))^2}{|x-y|^{d+\alpha}} dx dy \Big)^{1/2}.
 \end{array}
\end{equation}
Since  $\varphi \in C_0^\infty(\mathbb{R}^d )$, we obtain  using \eqref{M-1} and estimate
$|\varphi(y) - \varphi(x)| \le \big(\max |\nabla \varphi|\big)\, |y-x|$
$$
\frac{1}{\eps^{d+\alpha}} \int\limits_{G_2^\delta }
 p\big(\frac{x-y}{\eps}\big)
 (\varphi(y) - \varphi(x))^2 dx dy=
 \frac{1}{\eps^{d+\alpha}} \int\limits_{G_2^\delta\cap\{|x-y|<M\eps\} }
 p\big(\frac{x-y}{\eps}\big)
 (\varphi(y) - \varphi(x))^2 dx dy
$$
$$
+\frac{1}{\eps^{d+\alpha}} \int\limits_{G_2^\delta\cap\{|x-y|>M\eps\} }
 p\big(\frac{x-y}{\eps}\big)
 (\varphi(y) - \varphi(x))^2 dx dy
$$
$$
\leqslant \frac{C_1\eps^2}{\eps^{d+\alpha}}
\int\limits_{\{|z|<M\eps\} }
 p\big(\frac{z}{\eps}\big)dz+\beta_2 \int\limits_{G_2^\delta} \frac{ (\varphi(y) - \varphi(x))^2}{|x-y|^{d+\alpha}} dx dy
$$
$$
\leqslant 2 M^2 C_1 \eps^{2-\alpha} \int\limits_{|u|<M} p(u) du + \beta_2 C_1 \int\limits_{|z|<\delta} \frac{dz}{|z|^{d+\alpha-2}}
 \leqslant  2 M^2 C_1\eps^{2-\alpha}+ \beta_2 C_1 \frac{1}{2-\alpha} \delta^{2-\alpha};
$$
here $C_1=\|\varphi\|^2_{C^1(\mathbb R^d)}\big(\sup\{|x|\,:\, \varphi(x)\not=0\}\big)^d$, where
$$\|\varphi\|_{C^1(\mathbb R^d)} = \max \Big\{ \|\varphi\|_{C(\mathbb R^d)}, \; \|\partial_j \varphi\|_{C(\mathbb R^d)}, \, j =1, \ldots, d \Big\}.$$

Since  $\varphi \in C_0^\infty(\mathbb{R}^d )$, then for sufficiently small $\delta>0$ the integration over $G_3^\delta $
is reduced to the integration over the sets
% only domains %of the form $\{|x|> \delta^{-1}, \, |y| \le C\}$,
%where $\varphi(y) - \varphi(x) \neq 0$, give the contribution to the integral over $G_3^\delta $. For small enough $\delta$ such domain in $G_3^\delta$ can be describe as follows:
$$
\{|x|> \delta^{-1}-C, \, |y| \le C\} \quad \mbox{and } \quad \{|y|> \delta^{-1}-C, \, |x| \le C\},
$$
where $C$ is a constant that depends on the ${\rm supp} \, \varphi$.
In these domains inequality \eqref{M-1} holds, and we get
$$
%\int\limits_{G_2^\delta} \frac{ (\varphi(y) - \varphi(x))^2}{|x-y|^{d+\alpha}} dx dy = O(\delta^{2-\alpha}), \qquad
\frac{1}{\eps^{d+\alpha}} \int\limits_{G_3^\delta }
 p\big(\frac{x-y}{\eps}\big)
 (\varphi(y) - \varphi(x))^2 dx dy\leqslant
\beta_2 \int\limits_{G_3^\delta} \frac{ (\varphi(y) - \varphi(x))^2}{|x-y|^{d+\alpha}} dx dy
$$
$$
\le 4 \beta_2  C_1 \int\limits_{|z|> \frac12 \delta^{-1}} \frac{dz}{|z|^{d+\alpha}}
=  O(\delta^{\alpha}).
$$
Consequently, the last integral in \eqref{B2} tends to 0 as $\delta \to 0$ and $\eps\to0$, and we get
\begin{equation}\label{B3}
\lim\limits_{\delta \to 0} \; \lim\limits_{\eps \to 0}  \Big|  \frac{1}{\eps^{d+\alpha}} \int\limits_{G_2^\delta \cup  G_3^\delta}
 p\big(\frac{x-y}{\eps}\big) \Lambda \big(\frac{x}{\eps}, \frac{y}{\eps} \big) ( u^\eps(y) - u^\eps(x) )
 (\varphi(y) - \varphi(x))dx dy \Big| = 0.
\end{equation}

Using the same reasoning for the solution $u$ of equation \eqref{U-limit} with $ \Lambda^{\rm eff}(x,y) = \overline{\Lambda} \, k\big( \frac{x-y}{|x-y|} \big)$ we conclude that
\begin{equation}\label{B3-limit}
\lim\limits_{\delta \to 0} \;   \Big|  \int\limits_{G_2^\delta \cup  G_3^\delta}
\frac{ \overline{\Lambda}\, k \big( \frac{x-y}{|x-y|} \big) ( u(y) - u(x) )}{|x-y|^{d+\alpha}}
 (\varphi(y) - \varphi(x))dx dy \Big| = 0.
\end{equation}

\medskip

We are left with analysing the behaviour of the integral
\begin{equation}\label{G1}
 \frac{1}{\eps^{d+\alpha}} \int\limits_{G_1^\delta}
 p\big(\frac{x-y}{\eps}\big) \Lambda \big(\frac{x}{\eps}, \frac{y}{\eps} \big) ( u^\eps(y) - u^\eps(x) )
 (\varphi(y) - \varphi(x))dx dy
\end{equation}
as $\varepsilon \to 0$.
 Since the function $\Lambda(x,y)$ is periodic the family $\Lambda^{\eps}(x,y) = \Lambda \big(\frac{x}{\eps}, \frac{y}{\eps} \big)$ converges weakly in $L^2_{\rm loc}(\mathbb{R}^{2d})$ to the mean  $\overline{\Lambda}$ of the function $\Lambda(x,y)$: $\overline{\Lambda} = \int\limits_{\mathbb{T}^d} \int\limits_{\mathbb{T}^d} \Lambda(x,y) dx dy$. In the next lemma we prove that  %under our assumptions \eqref{M-1} - \eqref{M-3}  on $p(\cdot)$
 the family of functions $  \frac{1}{\eps^{d+\alpha}}  p\big(\frac{x-y}{\eps}\big) \Lambda \big(\frac{x}{\eps}, \frac{y}{\eps} \big)$  %{\color{red} for a subsequence}
 converges weakly in $L^2 (G_1^\delta)$ to the function $\, \frac{\overline{\Lambda} \,k ( \frac{x-y}{|x-y|} )}{|x-y|^{d+\alpha}}$, as $\varepsilon \to 0$.

\begin{lemma}\label{WeakC}
Assume that  $p(\cdot) \in L^1(\mathbb{R}^d)$ meets all the conditions  \eqref{bas_prop_p} - \eqref{M-3}, and $ \Lambda (x, y)$ is a symmetric periodic function satisfying condition \eqref{Lambda}. Then,  for any  $\Psi \in L^2(G_1^\delta)$,
%{\color{red} for a subsequence} the following convergence holds:
\begin{equation}\label{WeC}
 \frac{1}{\eps^{d+\alpha}} \int\limits_{G_1^\delta}
 p\big(\frac{x-y}{\eps}\big) \Lambda \big(\frac{x}{\eps}, \frac{y}{\eps} \big) \Psi(x,y)\, dx dy \ \to \ \overline{\Lambda} \int\limits_{G_1^\delta} \frac{k \big( \frac{x-y}{|x-y|} \big)}{|x-y|^{d+\alpha}}  \Psi(x,y)\, dx dy, % \quad \Psi \in L^2(G_1^\delta)
\end{equation}
as $\varepsilon \to 0$.
\end{lemma}

\begin{proof}
Since $\frac\delta\eps \to \infty$ as $\eps \to 0$, due to condition \eqref{M-1} we have
$$
\int\limits_{G_1^\delta}\bigg(\frac{p\big(\frac {x-y}\eps\big)}{\eps^{d+\alpha}}\bigg)^2dxdy\leqslant
\int\limits_{G_1^\delta}\frac{\beta_2^2}{|x-y|^{2d+2\alpha}}dxdy\leqslant C \delta^{-2d-2\alpha}.
$$
Therefore, for each $\delta>0$, the family $\{\eps^{-d-\alpha}p\big(\frac{x-y}\eps\big)\}$ is bounded
in $L^2(G_1^\delta)$.
Since the set of $C_0^\infty(G_1^\delta)$ functions is dense in $L^2(G_1^\delta)$, it is sufficient to  take in relation \eqref{WeC} a smooth
function  $\Psi(x,y)$ with a compact support in $G_1^\delta$. We show first using assumptions \eqref{M-2} - \eqref{M-3}  that
\begin{equation}\label{WeC-1}
 \frac{1}{\eps^{d+\alpha}} \int\limits_{G_1^\delta}
 p\big(\frac{x-y}{\eps}\big) \Lambda \big(\frac{x}{\eps}, \frac{y}{\eps} \big) \Psi(x,y)\, dx dy \ = \ \frac{(\overline{\Lambda}+ o(1))}{\eps^{d+\alpha}} \int\limits_{G_1^\delta}   p\big(\frac{x-y}{\eps}\big)  \Psi(x,y)\, dx dy, \quad \varepsilon \to 0,
\end{equation}
where $o(1)\to 0$ as $\eps\to0$.

Denote $I_k(\varepsilon) = \varepsilon k + \varepsilon {[-\frac12,\frac12]}^{2d}, \; k \in \mathbb{Z}^{2d}$, and let $(x_k, y_k)$ be the center point of the box $I_k(\varepsilon)$. The set of $k \in \mathbb{Z}^{2d}$ such that
$I_k(\varepsilon)$ has a non-empty intersection with $G_1^\delta$ is denoted by $\mathcal{J}_\delta(\eps)$.
 Then we  define the mean values of $p\big(\frac{x-y}\eps \big)$
over the cubes $I_k(\varepsilon)$ by
$$
\hat p_k=\eps^{-2d}
\int\limits_{I_k(\varepsilon)} p\Big(\frac{x-y}\eps\Big)dxdy = \int\limits_{k+{[-\frac12,\frac12]}^{2d}} p(x-y)dxdy
$$
and introduce the following piece-wise constant function
$$
\hat p_{\eps}(x,y)= \hat p_k, \quad\hbox{if }(x,y)\in I_k(\eps), \quad k\in\mathbb Z^{2d}.
$$
Now the integral on the left hand side of \eqref{WeC-1} can be written as follows:
\begin{equation}\label{WeC-2}
\begin{array}{l}
\displaystyle
 \frac{1}{\eps^{d+\alpha}} \int\limits_{G_1^\delta}
 p\big(\frac{x-y}{\eps}\big) \Lambda \big(\frac{x}{\eps}, \frac{y}{\eps} \big) \Psi(x,y)\, dx dy
\\[3mm] \displaystyle
= \frac{1}{\eps^{d+\alpha}} \sum_{k \in \mathcal{J}_\delta(\eps)} \int\limits_{I_k(\varepsilon)}
 p\big(\frac{x-y}{\eps}\big) \Lambda \big(\frac{x}{\eps}, \frac{y}{\eps} \big) \Psi(x_k,y_k)\, dx dy \big( 1 + o(1) \big)
 \\[3mm] \displaystyle
 =  \frac{1}{\eps^{d+\alpha}} \, \overline{\Lambda} \, \sum_{k \in \mathcal{J}_\delta(\eps)} \Psi(x_k,y_k) \, \hat p_k \, \varepsilon^{2d} \, \big( 1 + o(1) \big)
 \\[2mm] \displaystyle
  +  \frac{1}{\eps^{d+\alpha}}  \sum_{k \in \mathcal{J}_\delta(\eps)}  \Psi(x_k,y_k) \, \int\limits_{I_k(\varepsilon)}
\Big( p\big(\frac{x-y}{\eps}\big)-  \hat p_\eps(x,y)  \Big) \Lambda \big(\frac{x}{\eps}, \frac{y}{\eps} \big) \, dx dy \, \big( 1 + o(1) \big),
 \end{array}
\end{equation}
where $o(1)\to 0$ as $\eps\to0$.
Since $x, \, y \in G_1^\delta$, then $|z| = \big|\frac{x-y}{\varepsilon} \big|> \frac{\delta}{\varepsilon} \to \infty$ as $\varepsilon \to 0$. Thus, taking into account condition \eqref{M-3} with $K = 2\sqrt{d}$ we conclude that for any
$k\in\mathcal{J}_\delta(\eps)$ and for almost all $(x,y) \in I_k(\varepsilon) $ the inequality
\begin{equation}\label{WeC-3}
\big|  p\big(\frac{x-y}{\eps}\big)-  \hat p_\eps(x,y) \big|=\big|  p\big(\frac{x-y}{\eps}\big)-  \hat p_k \big|
\leqslant \phi_K \big(\frac\delta{2\eps}\big)p\big(\frac{x-y}{\eps}\big)
%As a consequence of condition \eqref{M-3}, for almost all $(x,y)\in I_k(\eps)$ we have
%$$
%\big|\hat p_{\eps}(x,y)-p\Big(\frac{x-y}\eps\Big)\big|\leqslant \phi_K().
%$$
%= | p(z_k + \gamma) - p(z_k)| \le \varphi(|z_k|) p(z_k),
\end{equation}
holds for each $\eps<(4d)^{-\frac12}\delta$. Indeed, if $\eps<(4d)^{-\frac12}\delta$ and $k\in\mathcal{J}_\delta(\eps)$, then for almost all $(x,y)\in I_k(\eps)$
we have
$$
\big|  p\big(\frac{x-y}{\eps}\big)-  \hat p_k \big|=\big|  p\big(\frac{x-y}{\eps}\big)-
\eps^{-2d}
\int\limits_{I_k(\varepsilon)} p\Big(\frac{\xi-\eta}\eps\Big)d\xi d\eta \big|
$$
$$
\leqslant \eps^{-2d}\int\limits_{I_k(\varepsilon)}\Big|p\big(\frac{x-y}{\eps}\big)-  p\Big(\frac{\xi-\eta}\eps\Big)\Big|d\xi d\eta \leqslant \phi_K \big(\frac\delta{2\eps}\big)p\big(\frac{x-y}{\eps}\big).
$$
 %where $\varphi(r) \to 0$ as $r \to \infty$.
 Consequently, the last sum in \eqref{WeC-2} can be estimated  as follows:
\begin{equation}\label{WeC-4}
\begin{array}{l}
\displaystyle
\Big| \frac{1}{\eps^{d+\alpha}}  \sum_{k \in \mathcal{J}_\delta(\eps)}  \Psi(x_k,y_k) \, \int\limits_{I_k(\varepsilon)}
\Big( p\big(\frac{x-y}{\eps}\big)-  \hat p_k\Big) \Lambda \big(\frac{x}{\eps}, \frac{y}{\eps} \big) \, dx dy  \Big|
\\[3mm] \displaystyle
 \le \ \frac{ \gamma_2}{\eps^{d+\alpha}} \phi_K \big(\frac\delta{2\eps}\big)  \sum_{k \in \mathcal{J}_\delta(\eps)} %\varepsilon^{2d} \,
| \Psi(x_k,y_k)| \int\limits_{I_k(\varepsilon)}  p\big(\frac{x-y}{\eps}\big)dxdy % \,  \varphi(|\frac{x_k-y_k}{\eps}|),
 \\[3mm] \displaystyle
 \leqslant \frac{ \gamma_2}{\eps^{d+\alpha}} \phi_K \big(\frac\delta{2\eps}\big) \|\Psi\|_{C(\mathbb R^d)}
  \int\limits_{G_1^\delta}  p\big(\frac{x-y}{\eps}\big)dxdy
\end{array}
\end{equation}
with $\phi_K(\frac\delta{2\eps}) \to 0$ as $\eps \to 0$.
Combining relations \eqref{WeC-4} and \eqref{WeC-2} we obtain
\begin{equation}\label{WeC-5}
\begin{array}{l}
\displaystyle
 \frac{1}{\eps^{d+\alpha}} \int\limits_{G_1^\delta}
 p\big(\frac{x-y}{\eps}\big) \Lambda \big(\frac{x}{\eps}, \frac{y}{\eps} \big) \Psi(x,y)\, dx dy
\\[3mm] \displaystyle
=  \frac{1}{\eps^{d+\alpha}} \,  \sum_{k \in \mathcal{J}_\delta(\eps)}  \Psi(x_k,y_k) \, \hat p_k \, \varepsilon^{2d} \, \big(  \overline{\Lambda} + o(1) \big)
\\[3mm] \displaystyle
 = \frac{\big(  \overline{\Lambda} + o(1) \big)}{\eps^{d+\alpha}}  \sum_{k \in \mathcal{J}_\delta(\eps)}  \int\limits_{I_k(\varepsilon)} \Psi(x,y) \, p\big(\frac{x-y}{\eps}\big)\, dx dy \; (1+ o(1))
 \\[3mm] \displaystyle
    =  \frac{\big(  \overline{\Lambda} + o(1) \big)}{\eps^{d+\alpha}} \, \int\limits_{G_1^\delta} p\big(\frac{x-y}{\eps}\big) \,  \Psi(x,y) \, dx dy.
\end{array}
\end{equation}
This yields  relation \eqref{WeC-1}.

On the other hand, condition \eqref{M-2}  implies that  the family of functions $\{  \frac{1}{\eps^{d+\alpha}} p\big(\frac{x-y}{\eps}\big) \}$ converges weakly to $ \frac{k \big( \frac{x-y}{|x-y|} \big)}{|x-y|^{d+\alpha}}$ in $L^2(G_1^\delta)$, see e.g. Proposition 8.3.1. \cite{Kol}. Indeed, since  for any fixed $\delta>0$
the family $\{\frac{1}{\eps^{d+\alpha}} p\big(\frac{x-y}{\eps}\big)\}$ is bounded in $L^2(G_1^\delta)$,
%the density of a measure for any $\varepsilon$, and the limiting measure also has a density w.r.t Lebesgue measure,
we only need to prove that for any $R_1$, $R_2 >R_1$ and an open set $\Omega$ on $S^{d-1}$ we have
% the measures of any open set of $G_1^\delta$ converge as $\varepsilon \to 0$.
$$
\begin{array}{l}
\displaystyle
\int\limits_{R_1<|x-y|<R_2} \; \int\limits_{\tilde z = \frac{x-y}{|x-y|} \in \Omega}  \frac{1}{\eps^{d+\alpha}} \, p\big(\frac{x-y}{\eps}\big) \, dx dy \\[3mm] \displaystyle
= \int\limits_{\frac{R_1}{\varepsilon}<|w|<\frac{R_2}{\varepsilon}} \; \int\limits_{\tilde z \in \Omega}  \frac{1}{\eps^{\alpha}} \, p(w) \, dw \, d \tilde z \ \to \
\int\limits_{R_1}^{R_2} \frac{ d r}{r^{\alpha+1}} \, \int\limits_{\tilde z \in \Omega} k(s) \, d s.
\end{array}
$$
This relation follows from \eqref{M-2}.
Thus, convergence \eqref{WeC} is proved.
\end{proof}

Combining the strong convergence of $u^\eps$ in $L^2_{loc}$ with a weak convergence \eqref{WeC} in Lemma \ref{WeakC} and relation \eqref{B3} we get
\begin{equation}\label{B5}
\begin{array}{l}
\displaystyle
 \frac{1}{\eps^{d+\alpha}} \int\limits_{\mathbb{R}^d} \int\limits_{\mathbb{R}^d}
 p\big(\frac{x-y}{\eps}\big) \Lambda \big(\frac{x}{\eps}, \frac{y}{\eps} \big) ( u^\eps(y) - u^\eps(x) )
 (\varphi(y) - \varphi(x))dx dy  \\
 \displaystyle
 \rightarrow \  \overline{\Lambda} \, \int\limits_{\mathbb{R}^d} \int\limits_{\mathbb{R}^d}
 \frac{ k\big( \frac{x-y}{|x-y|} \big) ( u(y) - u(x) )}{|x-y|^{d+\alpha}} (\varphi(y) - \varphi(x)) dx dy.
 \end{array}
 \end{equation}
%$$
%\int_{\mathbb{R}^d \times \mathbb{R}^d} \frac{\overline{\Lambda}\, \tilde k(s) \, (u(x) - u(y)) (\varphi(x) - %\varphi(y))}{|x-y|^{d+\alpha}} dx dy + \int_{\mathbb{R}^d}(m u \varphi - f \varphi) dx =0.
%$$
Since $\varphi$ is an arbitrary function from $C_0^\infty(\mathbb{R}^d)$, we conclude that $u$ is a solution of equation $-L^0 u + m u = f$.  Due to uniqueness of a solution of this equation, the whole sequence $u^{\eps}$ converges to $u$ in $L^2_{\rm loc}(\mathbb{R}^d)$ as $\varepsilon \to 0$.
%equation $-L^0 u + m u = f$ with
%\begin{equation}\label{L0}
%  L^0 u(x) =  \int\limits_{\mathbb{R}^d} \frac{ \overline{\Lambda} \, \tilde k(s) \, ( u(y) - %u(x))}{|x-y|^{d+\alpha}} dy \quad \mbox{with } \; \overline{\Lambda} = \int\limits_{\mathbb{T}^d \times %\mathbb{T}^d} \Lambda(\xi,\eta) d\xi d\eta.
%\end{equation}
\subsection{Convergence in $L^2(\mathbb{R}^d)$}\label{ss_conv}

It remains to justify the convergence in $L^2(\mathbb{R}^d)$.
We introduce a function $\varphi_L(x)$ as follows:
$$
\varphi_L(x)=\left\{
\begin{array}{ll}
0, &\hbox{if }|x|<L\\
\frac 1L(|x|-L),&\hbox{if }L\leqslant |x|\leqslant 2L,\\
1,&\hbox{otherwise}.
\end{array}
\right.
$$
Our goal is to show that $\|\varphi^\frac12_Lu^\eps\|_{L^2(\mathbb R^d)}\to 0$ as $L\to\infty$ uniformly in $\eps>0$.
To this end we multiply equation \eqref{Ue} by $\varphi_Lu^\eps$ and integrate the resulting relation over $\mathbb R^d$.
This yields
$$
\frac1{\eps^{d+\alpha}}\int\limits_{\mathbb R^{2d}}p\Big(\frac{x-y}\eps\Big)\big(u^\eps(x)-u^\eps(y)\big)
\big(\varphi_L(x)u^\eps(x)-\varphi_L(y)u^\eps(y)\big)dxdy
$$
\begin{equation}\label{int_ide}
+m\int\limits_{\mathbb R^d}\varphi_L(x)(u^\eps(x))^2dx=\int\limits_{\mathbb R^d}\varphi_L(x)f(x)u^\eps(x)dx.
\end{equation}
Clearly,
\begin{equation}\label{small_disc}
\bigg|  \int\limits_{\mathbb R^d}\varphi_L(x)f(x)u^\eps(x)dx\bigg|\leqslant \|u^\eps\|_{L^2(\mathbb R^d)}
\|\varphi_L f\|_{L^2(\mathbb R^d)}\leqslant C\|\varphi_L f\|_{L^2(\mathbb R^d)}\to 0,
\end{equation}
as $L\to\infty$. The first integral in \eqref{int_ide} can be rearranged as follows:
$$
\frac1{\eps^{d+\alpha}}\int\limits_{\mathbb R^{2d}}p\Big(\frac{x-y}\eps\Big)\big(u^\eps(x)-u^\eps(y)\big)
\big(\varphi_L(x)u^\eps(x)-\varphi_L(y)u^\eps(y)\big)dxdy
$$
\begin{equation}\label{fists_nonneg}
=\frac1{\eps^{d+\alpha}}\int\limits_{\mathbb R^{2d}}p\Big(\frac{x-y}\eps\Big)\varphi_L(x)\big(u^\eps(x)-u^\eps(y)\big)^2
dxdy
\end{equation}
$$
+\frac1{\eps^{d+\alpha}}\int\limits_{\mathbb R^{2d}}p\Big(\frac{x-y}\eps\Big)\big(u^\eps(x)-u^\eps(y)\big)\big(\varphi_L(x)-\varphi_L(y)\big)u^\eps(y)
dxdy.
$$
Integral \eqref{fists_nonneg} is non-negative. Let us estimate the second integral on the right-hand side.
$$
\bigg|\frac1{\eps^{d+\alpha}}\int\limits_{\mathbb R^{2d}}p\Big(\frac{x-y}\eps\Big)\big(u^\eps(x)-u^\eps(y)\big)\big(\varphi_L(x)-\varphi_L(y)\big)u^\eps(y)
dxdy\bigg|\leqslant
$$
$$
\leqslant
\bigg|\frac1{\eps^{d+\alpha}}\int\limits_{|x-y|<M \eps}p\Big(\frac{x-y}\eps\Big)\big(u^\eps(x)-u^\eps(y)\big)\big(\varphi_L(x)-\varphi_L(y)\big)u^\eps(y)
dxdy\bigg|+
$$
$$
+\bigg|\frac1{\eps^{d+\alpha}}\int\limits_{|x-y|>M \eps}p\Big(\frac{x-y}\eps\Big)\big(u^\eps(x)-u^\eps(y)\big)\big(\varphi_L(x)-\varphi_L(y)\big)u^\eps(y)
dxdy\bigg|=I_1+I_2.
$$
Using  the fact that $|\varphi_L(x)-\varphi_L(y)|\leqslant \frac1L|x-y|$, we obtain
$$
I_1\leqslant \Big(\frac1{\eps^{d+\alpha}}\int\limits_{\mathbb R^d}p\Big(\frac{x-y}\eps\Big)\big(u^\eps(x)-u^\eps(y)\big)^2
dxdy\Big)^\frac12\, \Big(\frac{M^2\eps^2}{L^2 \eps^{d+\alpha}}\int\limits_{|x-y|<M\eps}p\Big(\frac{x-y}\eps\Big)
(u^\eps(y))^2dxdy\Big)^\frac12
$$
$$
\leqslant CL^{-1}M\eps \Big(\frac1{\eps^{d+\alpha}}\int\limits_{|z|<M\eps}p\Big(\frac{z}\eps\Big)
(u^\eps(y))^2dzdy\Big)^\frac12\leqslant CL^{-1}M\eps^{1-\frac\alpha2}
$$
The integral $I_2$ admits the following estimate:
 $$
 I_2\leqslant
 \bigg(\frac1{\eps^{d+\alpha}}\int\limits_{|x-y|>M \eps}p\Big(\frac{x-y}\eps\Big)\big(u^\eps(x)-u^\eps(y)\big)^2
dxdy\bigg)^\frac12
\bigg(\beta_2\int\limits_{|x-y|>M \eps}\frac{\big(\varphi_L(x)-\varphi_L(y)\big)^2u^\eps(y)^2}{|x-y|^{d+\alpha}}
dxdy\bigg)^\frac12
 $$
 $$
 \leqslant C
 \Big(\int\limits_{\mathbb R^d}\frac{\big(\varphi_L(x)-\varphi_L(y)\big)^2 (u^\eps(y))^2}{|x-y|^{d+\alpha}}
dxdy\Big)^\frac12
 $$
Since $|\varphi_L(x)-\varphi_L(y)|\leqslant \min\{\frac1L|x-y|, 1\}$, we have
$$
\int\limits_{\mathbb R^d}\frac{\big(\varphi_L(x)-\varphi_L(y)\big)^2 (u^\eps(y))^2}{|x-y|^{d+\alpha}}
dxdy
\leqslant\int\limits_{|x-y|<L}\frac{|x-y|^2 (u^\eps(y))^2}{L^2|x-y|^{d+\alpha}}dxdy+
\int\limits_{|x-y|>L}\frac{ (u^\eps(y))^2}{|x-y|^{d+\alpha}}dxdy
$$
 $$
 \leqslant C\int\limits_{|z|<L}\frac{|z|^2 }{L^2|z|^{d+\alpha}}dz+
 \int\limits_{|z|>L}\frac1{|z|^{d+\alpha}}dz\leqslant C L^{-\alpha}
 $$
 Therefore,
 $$
 I_2\leqslant C L^{-\frac\alpha2}
 $$
 Combining this inequality with the estimate for $I_1$ we conclude that
 $$
\bigg|\frac1{\eps^{d+\alpha}}\int\limits_{\mathbb R^{2d}}p\Big(\frac{x-y}\eps\Big)\big(u^\eps(x)-u^\eps(y)\big)\big(\varphi_L(x)-\varphi_L(y)\big)u^\eps(y)
dxdy\bigg|\leqslant CL^{-\frac\alpha2}.
$$
 Considering \eqref{int_ide}, \eqref{small_disc} and the positiveness of integral \eqref{fists_nonneg} we finally deduce that
 $$
 \lim\limits_{L\to\infty}\sup\limits_{\eps\in(0,1]}\int_{\mathbb R^2}\varphi_L(x)(u^\eps(x))^2 dx=0.
 $$
 This yields the desired convergence of $u^\eps$ to $u$ in $L^2(\mathbb R^d)$ and
 completes the proof of Theorem.
\end{proof}

\end{document}